\documentclass[11pt]{article}
\usepackage{amsmath, amssymb, amsthm, CJK, esvect, tikz}
\def\h{0.577350269189} 

\usepackage[all]{xy}
\usepackage[colorlinks=false]{hyperref}
\usepackage[utf8]{inputenc}

\title{A Slight Improvement to the Colored B\'{a}r\'{a}ny's Theorem}
\author{Zilin Jiang \begin{CJK}{UTF8}{gkai}姜子麟\end{CJK}\footnote{Department of Mathematical Sciences, Carnegie Mellon University, Pittsburgh, PA 15213. Supported in part by U.S. taxpayers through NSF grants DMS-1201380 and DMS-1301548.}\footnotemark[1]}
\date{}

\newtheorem*{maintheorem}{Main Theorem}

\newtheorem{theorem}{Theorem}
\newtheorem{corollary}[theorem]{Corollary}

\renewcommand{\leq}{\leqslant}
\renewcommand{\geq}{\geqslant}
\newcommand{\abs}[1]{\left|#1\right|}
\newcommand{\norm}[1]{\left|\left|#1\right|\right|}
\newcommand{\<}{\langle}
\renewcommand{\>}{\rangle}
\newcommand{\RR}{\mathbb{R}}
\newcommand{\prob}[1]{\mathsf{Prob}\left(#1\right)}
\newcommand{\cl}{\mathrm{cl}}
\renewcommand{\int}{\mathrm{int}}
\newcommand{\CC}{\mathcal{C}}
\newcommand{\NN}{\mathbb{N}}
\newcommand{\KK}{\mathsf{K}}
\newcommand{\TT}{\mathcal{T}}
\newcommand{\KKK}{\mathcal{K}}
\newcommand{\LL}{\mathcal{L}}
\newcommand{\mm}{\mathbf{m}}
\newcommand{\cone}{\mathrm{cone}}
\newcommand{\conv}{\mathrm{conv}}
\newcommand{\modtwo}{\mathrm{mod2}}

\begin{document}

\maketitle

\begin{abstract}
Suppose $d+1$ absolutely continuous probability measures $m_0, \ldots, m_d$ on $\RR^d$ are given. In this paper, we prove that there exists a point of $\RR^d$ that belongs to the convex hull of $d+1$ points $v_0, \ldots, v_d$ with probability at least $\frac{2d}{(d+1)!(d+1)}$, where each point $v_i$ is sampled independently according to probability measure $m_i$.
\end{abstract}

\section{Introduction}

Let $P\subset \RR^d$ be a set of $n$ points. Every $d+1$ of them span a simplex, for a total of $n\choose {d+1}$ simplices. The point selection problem asks for a point contained in as many simplices as possible. Boros and F\"{u}redi \cite{MR771183} showed for $d=2$ that there always exists a point in $\RR^2$ contained in at least $\frac{2}{9}{n\choose 3}-O(n^2)$ simplices. A short and clever proof of this result was given by Bukh \cite{MR2240753}. B\'{a}r\'{a}ny \cite{MR676720} generalized this result to higher dimensions:

\begin{theorem}[B\'{a}r\'{a}ny \cite{MR676720}]\label{thm: first-selection}
  There exists a point in $\RR^d$ that is contained in at least $c_d{n\choose {d+1}}-O(n^d)$ simplices, where $c_d > 0$ is a constant depending only on the dimension $d$.
\end{theorem}

This general result, the B\'{a}r\'{a}ny's theorem, is also known as the first selection lemma. We will henceforth denote by $c_d$ the largest possible constant for which the B\'{a}r\'{a}ny's theorem holds true. Bukh, Matou\v{s}ek and Nivasch \cite{MR2579699} used a specific construction called the stretched grid to prove that the constant $c_2=\frac{2}{9}$ in the planar case found by Boros and F\"{u}redi \cite{MR771183} is the best possible. In fact, they proved that $c_d\leq \frac{d!}{(d+1)^{d}}$. On the other hand, B\'{a}r\'{a}ny's proof in \cite{MR676720} implies that $c_d\geq (d+1)^{-d}$, and Wagner \cite{wagner2003k} improved it to $c_d\geq \frac{d^2+1}{(d+1)^{d+1}}$.

Gromov \cite{MR2671284} further improved the lower bound on $c_d$ by topological means. His method gives $c_d\geq\frac{2d}{(d+1)(d+1)!}$. Matou\v{s}ek and Wagner \cite{matousek2011gromov} provided an exposition of the combinatorial component of Gromov's approach in a combinatorial language, while Karasev \cite{MR2891243} found a very elegant proof of Gromov's bound, which he described as a ``decoded and refined'' version of Gromov's proof. 

The exact value of $c_d$ has been the subject of ongoing research and is unknown, except for the planar case. Basit, Mustafa, Ray and Raza \cite{MR2742969} and successively Matou\v{s}ek and Wagner \cite{matousek2011gromov} improved the B\'{a}r\'{a}ny's theorem in $\RR^3$. Kr\'{a}l', Mach and Sereni \cite{MR2946458} used flag algebras from extremal combinatorics and managed to further improve the lower bound on $c_3$ to more than $0.07480$, whereas the best upper bound known is $0.09375$.

However, in this paper, we are concerned with a colored variant of the point selection problem. Let $P_0, \ldots, P_d$ be $d+1$ disjoint finite sets in $\RR^d$. A \emph{colorful simplex} is the convex hull of $d+1$ points each of which comes from a distinct $P_i$. For the colored point selection problem, we are concerned with the point(s) contained in many colorful simplices. Karasev proved:
\begin{theorem}[Karasev \cite{MR2891243}]\label{thm: karasev}
Given a family of $d+1$ absolutely continuous probability measures $\mm = (m_0, \ldots, m_d)$ on $\RR^d$, an $\mm$-simplex\footnote{An $\mm$-simplex is actually a simplex-valued random variable.} is the convex hull of $d+1$ points $v_0, \ldots, v_d$ with each point $v_i$ sampled independently according to probability measure $m_i$. There exists a point of $\RR^d$ that is contained in an $\mm$-simplex with probability $p_d \geq \frac{1}{(d+1)!}$. In addition, if two probability measures coincide, then the probability can be improved to $p_d \geq \frac{2d}{(d+1)(d+1)!}$.
\end{theorem}
By a standard argument which we will provide immediately, a result on the colored point selection problem follows:
\begin{corollary}\label{cor: karasev}
  If $P_0, \ldots, P_d$ each contains $n$ points, then there exists a point that is contained in at least $\frac{1}{(d+1)!}\cdot n^{d+1}$ colorful simplices.
\end{corollary}

Our result drops the additional assumption in theorem \ref{thm: karasev}, hence improves corollary \ref{cor: karasev}:
\begin{maintheorem}\label{main theorem}
  There is a point in $\RR^d$ that belongs to an $\mm$-simplex with probability $p_d \geq \frac{2d}{(d+1)(d+1)!}$.
\end{maintheorem}

\begin{corollary}\label{thm: discrete}
  There exists a point that is contained in at least $\frac{2d}{(d+1)(d+1)!}\cdot n^{d+1}$ colorful simplices.
\end{corollary}

\begin{figure}[h]
  \centering
  \begin{tikzpicture}[thick,scale=0.6,inner sep=1mm]
    \definecolor{red}{rgb}{.859375,.265625,.21484375}
    \definecolor{green}{rgb}{.0625,.61328125,.34765625}
    \definecolor{blue}{rgb}{.26171875,.51953125,.95703125}
    \coordinate (A1) at (9.9,0.8);
    \coordinate (A2) at (9.0,7.8);
    \coordinate (A3) at (12.6,3.2);
    \coordinate (B1) at (8.1,5.2);
    \coordinate (B2) at (2.1,5.5);
    \coordinate (B3) at (0.1,7.7);
    \coordinate (C1) at (6.7,7.0);
    \coordinate (C2) at (14.2,8.3);
    \coordinate (C3) at (6.6,2.4);
    \coordinate (X) at (6.2,6.2);
    \draw[fill=gray,opacity=0.1] (A3) -- (B3) -- (C1) -- cycle;
    \draw[fill=gray,opacity=0.1] (A3) -- (B3) -- (C2) -- cycle;
    \draw[fill=gray,opacity=0.1] (A1) -- (B3) -- (C1) -- cycle;
    \draw[fill=gray,opacity=0.1] (A1) -- (B3) -- (C2) -- cycle;
    \draw[fill=gray,opacity=0.1] (A1) -- (B2) -- (C1) -- cycle;
    \draw[fill=gray,opacity=0.1] (A2) -- (B2) -- (C3) -- cycle;
    \node[circle,draw=red!60,fill=red!20] at (A1) {};
    \fill[fill=red!60] (A1) circle (2pt);
    \node[circle,draw=red!60,fill=red!20] at (A2) {};
    \fill[fill=red!60] (A2) circle (2pt);
    \node[circle,draw=red!60,fill=red!20] at (A3) {};
    \fill[fill=red!60] (A3) circle (2pt);
    \node[circle,draw=green!60,fill=green!20] at (B1) {};
    \fill[fill=green!60] (8,5.2) circle (2pt);
    \fill[fill=green!60] (8.2,5.2) circle (2pt);
    \node[circle,draw=green!60,fill=green!20] at (B2) {};
    \fill[fill=green!60] (2,5.5) circle (2pt);
    \fill[fill=green!60] (2.2,5.5) circle (2pt);
    \node[circle,draw=green!60,fill=green!20] at (B3) {};
    \fill[fill=green!60] (0,7.7) circle (2pt);
    \fill[fill=green!60] (0.2,7.7) circle (2pt);
    \node[circle,draw=blue!60,fill=blue!20] at (C1) {};
    \fill[fill=blue!60] (6.6,6.95) circle (2pt);
    \fill[fill=blue!60] (6.8,6.95) circle (2pt);
    \fill[fill=blue!60] (6.7,7.1) circle (2pt);
    \node[circle,draw=blue!60,fill=blue!20] at (C2) {};
    \fill[fill=blue!60] (14.1, 8.25) circle (2pt);
    \fill[fill=blue!60] (14.3, 8.25) circle (2pt);
    \fill[fill=blue!60] (14.2, 8.4) circle (2pt);
    \node[circle,draw=blue!60,fill=blue!20] at (C3) {};
    \fill[fill=blue!60] (6.5, 2.35) circle (2pt);
    \fill[fill=blue!60] (6.6, 2.5) circle (2pt);
    \fill[fill=blue!60] (6.7, 2.35) circle (2pt);
    \node[rectangle,draw=black!60,fill=black!20] at (X) {};
  \end{tikzpicture}
  \caption{$3$ red points, $3$ green points and $3$ blue points are placed in the plane. The point marked by a square is contained in 6 ($=\frac{2}{9}\cdot 3^3$) colorful triangles.}
\end{figure}
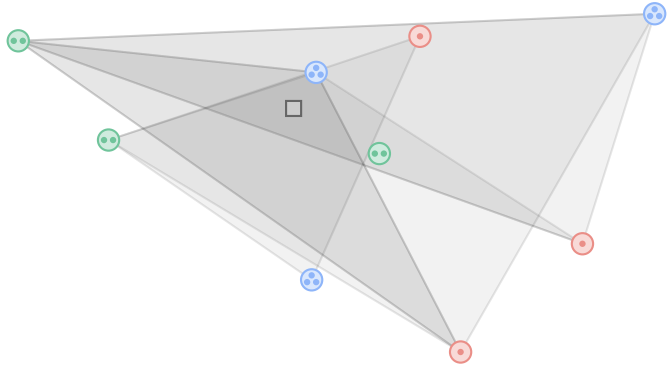

\begin{proof}[Proof of corollary \ref{thm: discrete} from the main theorem]
  Given $d+1$ sets $P_0, \ldots, P_d$ in $\RR^d$ each of which contains $n$ points. Let $\Psi\colon \RR^{d}\to\RR$ be the \emph{bump function} defined by $\Psi(x_1,\ldots,x_d)=\prod_{i=1}^d\psi(x_i)$, where $\psi(x)={e^{-1/(1-x^2)}}\mathbf{1}_{\abs{x}<1}$, and set $\Psi_n(x_1,\ldots,x_d)=n^{d}\Psi(nx_1,\ldots,nx_d)$ for $n\in\NN$. It is a standard fact that $\Psi$ and $\Psi_n$ are absolutely continuous probability measures supported on $[-1,1]^{d}$ and $[-1/n,1/n]^{d}$ respectively.
  
  For each $n\in\NN$ and $0\leq k\leq d$, define $m_k^{(n)}(x):=\frac{1}{n}\sum_{p\in P_k}\Psi_n(x-p)$ for $x\in\RR^d$. Note that $m_k^{(n)}$ is an absolutely continuous probability measure supported on the Minkowski sum of $P_k$ and $[-1/n, 1/n]^d$. Let $\mm^{(n)}$ be the family of $d+1$ probability measures $m_0^{(n)},\ldots, m_d^{(n)}$. By the main theorem, there is a point $p^{(n)}$ of $\RR^d$ that belongs to an $\mm^{(n)}$-simplex with probability at least $\frac{2d}{(d+1)(d+1)!}$.
  
  Because no point in a certain neighborhood of infinity is contained in any $\mm^{(n)}$-simplex, the set $\{p^{(n)}:n\in\NN\}$ is bounded, and consequently the set has a limit point $p$. Suppose $p$ is contained in $N$ colorful simplices. Let $\epsilon > 0$ be the distance from $p$ to all the colorful simplices that do not contain $p$. Choose $n$ large enough such that $1/n\ll\epsilon$ and $\abs{p^{(n)}-p}\ll\epsilon$. By the choice of $n$, if $p$ is not contained in a colorful simplex spanned by $v_0, \ldots, v_d$, then $p^{(n)}$ is not contained the convex hull of $v_0', \ldots, v_d'$ for all $v_i'\in v_i+[-1/n,1/n]^d$. This implies that the probability that $p^{(n)}$ is contained in an $\mm^{(n)}$-simplex is at most $\frac{N}{n^{d+1}}$. Hence $p$ is the desired point contained in $N\geq \frac{2d}{(d+1)(d+1)!}\cdot n^{d+1}$ colorful simplices.
\end{proof}

Readers who are familiar with Karasev's work \cite{MR2891243} would notice that our proof of the main theorem heavily relies on his arguments. The author is deeply in debt to him.

\section{Proof of the Main Theorem}

In this section, we provide the proof of the main theorem. The topological terms in the proof are standard, and can be found in \cite{matousek2003using}. In addition to the notion of an $\mm$-simplex, in the proof, we will often refer to an $(m_k,\ldots, m_d)$-face which means the convex hull of $d-k+1$ points $v_k,\ldots,v_d$ with each point $v_i$ sampled independently according to probability measure $m_i$. An $\mm$-simplex and an $(m_k,\ldots,m_d)$-face are both set-valued random variables.

\begin{proof}[Proof of the main theorem]

To obtain a contradiction, we suppose that for any point $v$ in $\RR^d$, the probability that $v$ belongs to an $\mm$-simplex is less than $p_d := \frac{2d}{(d+1)(d+1)!}$. Since this probability, as a function of point $v$, is continuous and uniformly tends to $0$ as $v$ goes to infinity, there is an $\epsilon > 0$ such that $v$ is contained in an $\mm$-simples with probability at most $p_d-\epsilon$ for all $v$ in $\RR^d$.

Let $S^d := \RR^d\cup\{\infty\}$ be the one-point compactification of the Euclidean space $\RR^d$. Take $\delta = \epsilon/d$. Choose a finite triangulation\footnote{A triangulation $\TT$ of a topological space $X$ is a simplicial complex $\KK$, homeomorphic to $X$, together with a homeomorphism $h\colon\norm{\KK}\to X$. Since the finite triangulation of interest is an extension of the triangulation of a $d$-simplex $X$ in $\RR^d$ and $h$ is an identity map, we will freely use topological notions such as ``a $k$-face (as a subset of $S^d$)'' instead of ``the image of a $k$-face in $\KK$ under $h$''. With such abuse of language, we can avoid going back and forth between the simplicial complex and the topological space.} $\TT$ of $S^d$ with one of the $d$-simplices containing $\infty$ such that for $0 < k \leq d$, any $k$-face of $\TT$ intersects an $(m_k,\ldots,m_d)$-face with probability less than $\delta$ and that the measure of any $d$-face of $\TT$ under $\left(m_{d-1}+m_d\right)/2$ is less than $\delta$. This can be done by taking a sufficiently fine triangulation of $S^2$ with one $d$-simplex having $\infty$ in its relative interior.

\begin{figure}
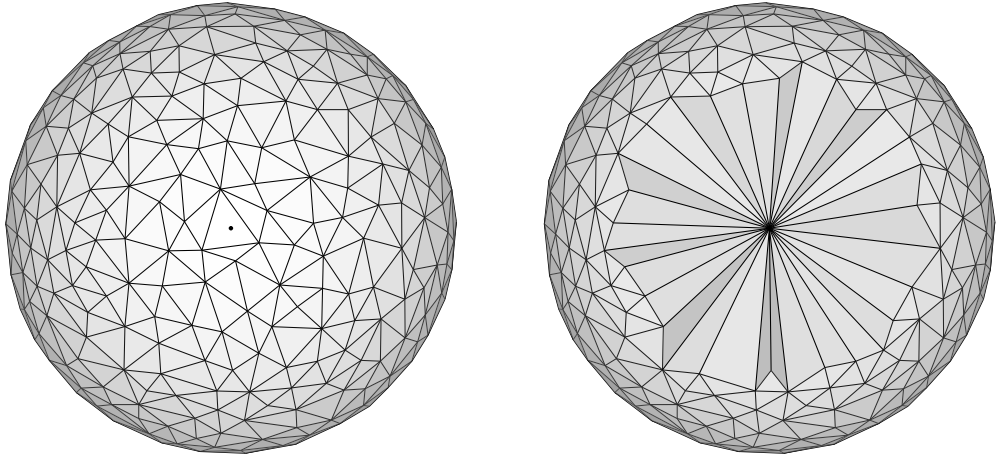

  \centering

  \caption{The bird's-eye view of a triangulation of $S^2$ with a $2$-simplex containing $\infty$ and the cone over part of the triangulation.}
\end{figure}

We use $\cone(\cdot)$ as the cone functor\footnote{The cone over a space $X$ is the quotient space $\cone(X) := \left(X\times[0,1]\right)/\left(X\times\{1\}\right)$. The apex is the equivalence class $\{(x,1): x\in X\}$.} with apex $O$. A triangulation $\TT$ of $S^d$ naturally extends to a triangulation $\cone(\TT)$ of $\cone(S^d)$. We denote the $k$-skeleton\footnote{The $k$-skeleton of a simplicial complex $\Delta$ consists of all simplices of $\Delta$ of dimension at most $k$.} of $\TT$ and $\cone(\TT)$ by $\TT^{\leq k}$ and $\cone(\TT)^{\leq k}$ respectively.

We are going to define a continuous map $f\colon \cone(\TT)^{\leq d}\to S^d$. Put $f(x)=x$ for all $x\in S^d = \norm{\TT} \subset \norm{\cone(\TT)^{\leq d}}$, and set $f(O) = \infty$. We proceed to define $f$ on $\cone(\sigma)$ for all the $k$-faces $\sigma$ of $\TT$ inductively on dimension $k$ of $\sigma$ while we maintain the property that the image of the boundary of $\cone(\sigma)$ under $f$, that is $f(\partial \cone(\sigma))$, intersects an $(m_k,\ldots, m_d)$-face with probability at most $(k+1)!(p_d-\epsilon+k\delta)$. We say $f$ is \emph{economical} over a $k$-face $\sigma$ of $\TT^{\leq{d-1}}$ if $f$ and $\sigma$ satisfy the above property. Unlike Karasev \cite{MR2891243}, our inductive construction of $f$ follows the same pattern until $k=d-2$ instead of $d-1$. The main innovation of this proof is a different construction for $k=d-1$, which enables us to remove the additional assumption in theorem \ref{thm: karasev}.

Note that for any $0$-face $\sigma$ in $\TT$, $f(\partial\cone(\sigma))=f(\{\sigma, O\})=\{\sigma, \infty\}$. According to the assumption at the beginning of the proof, $f(\partial\cone(\sigma))$ intersects an $(m_0,\ldots, m_d)$-face, that is, an $\mm$-simplex, with probability at most $p_d-\epsilon$. Therefore $f$ is economical over $0$-faces of $\TT$. This finishes the first step.
  
Suppose $f$ is already defined on $\cone(\TT)^{\leq k}$ and it is economical over $k$-faces of $\TT$. We are going to extend the domain of $f$ to $\cone(\TT)^{\leq k+1}$. Indeed, we only need to define $f$ on $\cone(\sigma)$ for every $k$-face $\sigma$ of $\TT$.

Take any $k$-face $\sigma$ of $\TT$. Suppose convex hull of $v_k, \ldots, v_d$, denoted by $\conv(v_k, \ldots, v_d)$, is an $(m_k,\ldots,m_d)$-face. Notice that the following statements are equivalent:
\begin{itemize}
  \item $f(\partial\cone(\sigma))$ intersects $\conv(v_k, \ldots, v_d)$;
  \item for some $v\in f(\partial\cone(\sigma))$, the ray with initial point $v$ in the direction $\vv{v_kv}$ intersects $\conv(v_{k+1},\ldots,v_d)$.
\end{itemize}
We call the union of such rays the \emph{shadow} of $f(\partial\cone(\sigma))$ centered at $v_k$. Since $f$ is economical over $\sigma$, the probability for an $(m_k, \ldots, m_d)$-face to meet $f(\partial\cone(\sigma))$ is at most $(k+1)!(p_d-\epsilon+k\delta)$, and so there exists $v_k^\sigma\in\RR^d$ such that the shadow of $f(\partial\cone(\sigma))$ centered at $v_k^\sigma$ intersects $\conv(v_{k+1},\ldots,v_d)$ with probability at most $(k+1)!(p_d-\epsilon+k\delta)$.

Now, we define $f$ on $\cone(\sigma)$. First, let $g$ be the homeomorphism from $\cone(\sigma)$ onto the cone over $\partial\cone(\sigma)$ with apex $c$ such that $g$ is an identity on $\partial\cone(\sigma)$. This can be done because $\cone(\sigma)$ is homeomorphic to a $(k+1)$-simplex $\Delta$ and it is easy to find a homeomorphism from $\Delta$ to $\cone(\partial\Delta)$ that keeps $\partial\Delta$ fixed.

\begin{figure}[h]
  \centering
  \begin{tikzpicture}[scale=1]
    \draw (1,0) -- (0,1);
    \node[circle, fill=black, outer sep=2pt, inner sep = 1pt, label=right:$e_0$] at (1,0) {};
    \node[circle, fill=black, outer sep=2pt, inner sep = 1pt, label=above:$e_1$] at (0,1) {};
  \end{tikzpicture}\quad\quad\quad
  \begin{tikzpicture}[scale=1]
    \node[circle, fill=black, outer sep=2pt, inner sep = 1pt, label=right:$e_0$] at (1,0) {};
    \node[circle, fill=black, outer sep=2pt, inner sep = 1pt, label=above:$e_1$] at (0,1) {};
  \end{tikzpicture}\quad\quad\quad
  \begin{tikzpicture}[scale=1]
    \draw (0,1) -- (0,0) -- (1,0);
    \node[circle, fill=black, outer sep=2pt, inner sep = 1pt, label=right:$e_0$] at (1,0) {};
    \node[circle, fill=black, outer sep=2pt, inner sep = 1pt, label=above:$e_1$] at (0,1) {};
    \node[circle, fill=black, outer sep=2pt, inner sep = 1pt, label=-145:$c$] at (0,0) {};
  \end{tikzpicture}\quad\quad\quad
  \begin{tikzpicture}[scale=1]
    \draw (0,1) -- (0,0) -- (1,0) -- cycle;
    \node[circle, fill=black, outer sep=2pt, inner sep = 1pt, label=right:$e_0$] at (1,0) {};
    \node[circle, fill=black, outer sep=2pt, inner sep = 1pt, label=above:$e_1$] at (0,1) {};
    \node[circle, fill=black, outer sep=2pt, inner sep = 1pt, label=-145:$c$] at (0,0) {};
    \draw[->] (0.45,0.45) -- (0.05,0.05);
    \draw[->] (0.62,0.29) -- (0.38,0.05);
    \draw[->] (0.29,0.62) -- (0.05,0.38);
    \draw[->] (0.79,0.13) -- (0.71,0.05);
    \draw[->] (0.13,0.79) -- (0.05,0.71);
  \end{tikzpicture}  
  \caption{An illustration of an $1$-simplex $\Delta$, $\partial\Delta$, $\cone(\partial\Delta)$ and a homeomorphism from $\Delta$ to $\cone(\partial\Delta)$.}
\end{figure}
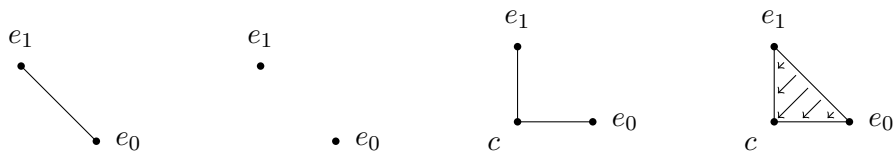

Next, note that every point $w$ in $\cone(\sigma)$ except $c$ is on a line segment $[v,c)$ for a unique point $v$ on $\partial\cone(\sigma)$. If $t=\overline{vw}/\overline{wc}\in[0,\infty)$, then put $h(w)=\vv{f(v)}+t\cdot\vv{v_k^\sigma f(v)}$. In addition, set $h(c)=\infty$. The function $h$ maps $[v, c)$ onto $[f(v), v_k^\sigma)$ linearly and then takes the inversion centered at $v_k^\sigma$ with radius $\overline{v_k^\sigma f(v)}$ so that $[f(v),v_k^\sigma)$ gets mapped onto the ray with the initial point $f(v)$ in the direction $\vv{v_k^\sigma f(v)}$. Evidently, $h$ is a continuous map from $\cone(\partial\cone(\sigma))$ onto the shadow of $f(\partial\cone)$ centered at $v_k^\sigma$ that coincides with $f$ on $\partial\cone(\sigma)$.

\begin{figure}[h]
  \centering
  \begin{tikzpicture}
    \begin{scope}[xshift=-50]
      \draw (0,0) -- (-1,0);
      \draw (0,-\h) -- (0,\h);
      \node[circle, fill=black, outer sep=2pt, inner sep = 1pt, label=right:$v$] at (0,0) {};
      \node[circle, fill=black, outer sep=2pt, inner sep = 1pt, label=left:$c$] at (-1,0) {};
      \node[circle, fill=black, outer sep=2pt, inner sep = 1pt, label=below:$w$] at (-0.5,0) {};
      \node[label=below:$\partial\cone(\sigma)$] at (0,-1) {};
      \draw[fill=gray,opacity=0.1] (0,-\h) -- (0,\h) -- (-1,0) -- cycle;
      \draw[dashed, ->] (-0.4,0.1) to [bend left] (2.9, 0.1);
    \end{scope}
    \begin{scope}[xshift=50]
      \draw[dashed] (-1,0) -- (2,0);
      \draw (0,-\h) -- (0,\h);
      \node[circle, fill=black, outer sep=2pt, inner sep = 1pt, label=left:$v^\sigma_1$] at (-1,0) {};
      \node[circle, fill=black, outer sep=2pt, inner sep = 1pt, label=-45:$f(v)$] at (0,0) {};
      \node[circle, fill=black, outer sep=2pt, inner sep = 1pt, label=-45:$h(w)$] at (1,0) {};
      \node[circle, fill=black, outer sep=2pt, inner sep = 1pt, label=below:$w'$] at (-0.5,0) {};
      \node[label=below:$f(\partial\cone(\sigma))$] at (0,-1) {};
      \draw[dashed, ->] (-0.4,0.1) to [bend left] (0.9, 0.1);
      \draw[fill=gray,opacity=0.1] (0,-\h) -- (0,\h) -- (2,3*\h) -- (2,-3*\h) -- cycle;
      \end{scope}
  \end{tikzpicture}  
  \caption{The illustration shows a cone over part of $\partial\cone(\sigma)$ with apex $c$ and a point $v$ on the boundary, and how a point $w$ on the line segment $[v, c)$ are mapped under $h$.}
\end{figure}
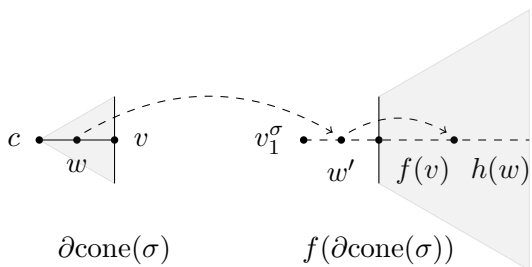

Define $f$ on $\cone(\sigma)$ to be the composition of $g$ and $h$: \[
\xymatrix{
\partial\cone(\sigma)\ar@{^{(}->}[d]\ar[r]^= & \partial\cone(\sigma)\ar@{^{(}->}[d]\ar[r]^f & f(\partial\cone(\sigma))\ar@{^{(}->}[d]\\
\cone(\sigma) \ar[r]^(0.4)g & \cone\left(\partial\cone(\sigma)\right)\ar[r]^(0.3)h & \text{the shadow of }f(\partial\cone(\sigma))\text{ centered at }v_k^\sigma.}
\smallskip  
\]
According to the commutative diagram above, $f$ is well-defined on $\cone(\sigma)$ in the sense that it is compatible with its definition on $\cone(\TT)^{\leq k}$. We use the phrase ``fill in the boundary of $\cone(\sigma)$ against the center $v_k^\sigma$'' to represent the above process that extends the domain of $f$ from $\partial\cone(\sigma)$ to $\cone(\sigma)$.

To complete the inductive step, we must demonstrate that $f$ is economical over $(k+1)$-faces of $\TT$. Pick any $(k+1)$-face $\tau$ of $\TT$. Let $\sigma_0, \ldots, \sigma_{k+1}$ be the $k$-faces of $\tau$. Observing that $f(\partial\cone(\tau)) = f(\tau\cup\cone(\partial\tau))=\tau\cup f(\cone(\sigma_0))\cup\ldots\cup f(\cone(\sigma_{k+1}))$ and that $f(\cone(\sigma_i))$ is the shadow of $f(\partial\cone(\sigma_i))$ centered at $v_k^{\sigma_i}$ which intersects an $(m_{k+1},\ldots,m_d)$-face with probability at most $(k+1)!(p_d-\epsilon+k\delta)$, we obtain that the probability for an $(m_{k+1}, \ldots, m_d)$-face to intersect $f(\partial\cone(\tau))$ is dominated by $\delta+(k+2)(k+1)!(p_d-\epsilon+k\delta)\leq (k+2)!(p_d-\epsilon+(k+1)\delta)$.

We have so far defined a continuous map $f$ on $\cone(\TT)^{\leq d-1}$ such that for any $(d-1)$-face $\sigma$ of $\TT$ the probability for an $(m_{d-1}m_d)$-face to intersect $D:=f(\partial\cone(\sigma))$ is at most $d!(p_d-\epsilon+(d-1)\delta)$. We write $f(X)\modtwo := \{y\in f(X):\abs{f^{-1}(y)\cap X}=1\pmod 2\}$ for the set of points in $f(X)$ whose fibers in $X$ have an odd number of points. Set $\bar{m}:=(m_{d-1}+m_d)/2$. We are going to define $f$ on $\cone(\sigma)$ such that $\bar{m}\left(f(\cone(\sigma)) \modtwo\right)$ is less than $\frac{1-\delta}{d+1}$.

Fix a point $s$ in $\RR^d\backslash D$. For any point $t$ in $\RR^d\backslash D$, if a generic piecewise linear path from $s$ to $t$ intersects with $D$ an odd number of times, then put $t$ in $B$, otherwise put it in $A$. Here the number of intersections of a piecewise linear path $L$ and $D$ might not be the cardinality of $L\cap D$. Instead, the number of intersections is precisely $\sum_{x\in L\cap D}\abs{f^{-1}(x)\cap \partial\cone(\sigma)}$, that is, it takes the multiplicity into account. Thus we have partitioned $\RR^d\backslash D$ into $A$ and $B$ such that any generic piecewise linear path from a point in $A$ to a point in $B$ meets $D$ an odd number of times. Suppose $a:=m_{d-1}(A)$, $b:=m_d(A)$ and $x:=\bar{m}(A)=(a+b)/2$. The probability that an $(m_{d-1}m_d)$-face intersects with $D$ is at least $a(1-b)+(1-a)b$. Hence $a(1-b)+(1-a)b < d!(p_d-\epsilon+(d-1)\delta) <2\left(\frac{1-\delta}{d+1}\right)\left(1-\frac{1-\delta}{d+1}\right)$. Because $a(1-b)+(1-a)b=(a+b)-2ab\geq (a+b)-(a+b)^2/2 = 2x(1-x)$, either $x$ or $1-x$ is less than $\frac{1-\delta}{d+1}$. In other words, one of $\bar{m}(A)$ and $\bar{m}(B)$ is less than $\frac{1-\delta}{d+1}$. We may assume that $\bar{m}(B) < \frac{1-\delta}{d+1}$.

Fix a point $c\in A$. Again, we fill in the boundary of $\cone(\sigma)$ against the center $c$. For any generic point $x\in A$, the line segment $[c,x]$ intersects with $D$ an even number of times. For every $v$ on $\partial\cone(\sigma)$, the ray with the initial point $f(v)$ in the direction $\vv{cf(v)}$ covers $x$ once if and only if the line segment $[c,x]$ intersects with $D$ at $f(v)$. Because $f(\cone(\sigma))$ is the union of such rays, the number of times that $x$ is covered by $f(\cone(\sigma))$ is exactly the number of intersections between $[c,x]$ and $D$. This implies that $x$ is not in $f(\cone(\sigma))\modtwo$. Therefore $f(\cone(\sigma))\modtwo$ is a subset of $B\cup D$ almost surely. Noticing that $\bar{m}(D)=0$, the extension of $f$ has the desired property $\bar{m}\left(f(\cone(\sigma))\modtwo\right) < \frac{1-\delta}{d+1}$.

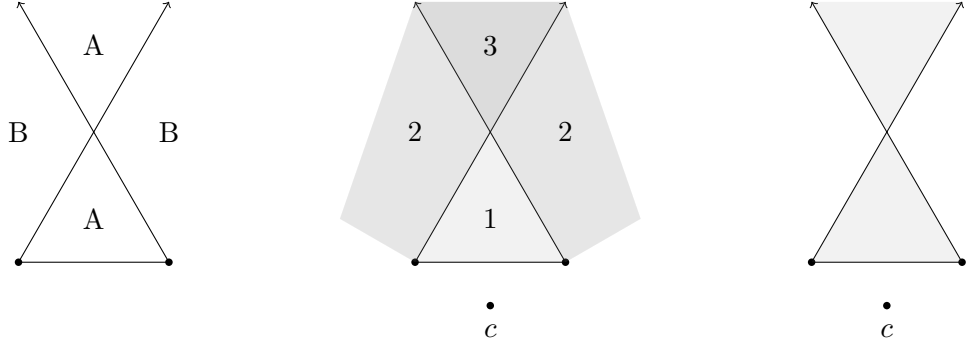
\begin{figure}
  \centering
  \begin{tikzpicture}[scale=1]
    \begin{scope}[xshift = -150]
      \draw[<->] (1, 2/\h) -- (-1,0) -- (1,0) -- (-1, 2/\h);
      \node[circle, fill=black, outer sep=2pt, inner sep = 1pt] at (1,0) {};
      \node[circle, fill=black, outer sep=2pt, inner sep = 1pt] at (-1,0) {};
      \node at (-1,3*\h) {B};
      \node at (1,3*\h) {B};
      \node at (0,\h) {A};
      \node at (0,5*\h) {A};
    \end{scope}
    \begin{scope}
      \draw[<->] (1, 2/\h) -- (-1,0) -- (1,0) -- (-1, 2/\h);
      \node[circle, fill=black, outer sep=2pt, inner sep = 1pt] at (1,0) {};
      \node[circle, fill=black, outer sep=2pt, inner sep = 1pt] at (-1,0) {};
      \node[circle, fill=black, outer sep=2pt, inner sep = 1pt, label=below:$c$] at (0,-\h) {};
      \fill[color=gray, opacity=0.1] (-1, 0) -- (-2,\h) -- (-1, 2/\h) -- (1,2/\h) -- cycle;
      \fill[color=gray, opacity=0.1] (1, 0) -- (2,\h) -- (1, 2/\h) -- (-1,2/\h) -- cycle;
      \fill[color=gray, opacity=0.1] (1, 0) -- (2,\h) -- (1, 2/\h) -- (-1,2/\h) -- (-2, \h) -- (-1, 0) -- cycle;
      \node at (-1,3*\h) {2};
      \node at (1,3*\h) {2};
      \node at (0,\h) {1};
      \node at (0,5*\h) {3};
    \end{scope}
    \begin{scope}[xshift=150]
      \draw[<->] (1, 2/\h) -- (-1,0) -- (1,0) -- (-1, 2/\h);
      \node[circle, fill=black, outer sep=2pt, inner sep = 1pt] at (1,0) {};
      \node[circle, fill=black, outer sep=2pt, inner sep = 1pt] at (-1,0) {};
      \node[circle, fill=black, outer sep=2pt, inner sep = 1pt, label=below:$c$] at (0,-\h) {};
      \fill[color=gray, opacity=0.1] (-1, 0) -- (0,3*\h) -- (1, 0) -- cycle;
      \fill[color=gray, opacity=0.1] (0, 3*\h) -- (-1,2/\h) -- (1, 2/\h) -- cycle;
    \end{scope}
  \end{tikzpicture}  
  \caption{An illustration of the partition, the result of filling in against $c$, and $f(\cone(\sigma))\modtwo$.}
\end{figure}

Pick any $d$-face $\tau$ of $\TT$. Suppose the $(d-1)$-faces of $\tau$ are $\sigma_0,\ldots, \sigma_d$. By a parity argument, we have 
\begin{eqnarray*}
f(\partial\cone(\tau))\modtwo &=& \left[\tau\cup f(\cone(\sigma_0))\cup\ldots\cup f(\cone(\sigma_d))\right]\modtwo\\
&\subset&\tau\cup f(\cone(\sigma_0))\modtwo\cup\ldots\cup f(\cone(\sigma_d))\modtwo.
\end{eqnarray*}
Therefore $\bar{m}\left(f(\partial\cone(\tau))\modtwo\right)$ is less than $\delta+(d+1)\frac{1-\delta}{d+1}=1$, and so the degree of $f$ on $\partial\cone(\tau)$, denoted by $\deg\left(f,{\partial\cone(\tau)}\right)$, is even. Because \[
  \sum_{\tau}\deg\left(f,{\partial\cone(\tau)}\right)= 2\sum_{\sigma}\deg\left(f,{\cone(\sigma)}\right) + \deg\left(f,{\TT}\right) = \deg\left(f,{\TT}\right) \pmod 2,
\]
where the first sum and the second sum are over all $d$-faces and all $(d-1)$-faces of $\TT$ respectively, we know that $\deg\left(f,{\TT}\right)$ is even, which contradicts with the fact that $f$ is identity on $\TT$.
\end{proof}

\subsection*{Acknowledgment}
The author would like to thank Boris Bukh for guidance and fruitful discussions on the B\'{a}r\'{a}ny's theorem. This article would not have been possible without his support. The author is also grateful to Roman Karasev and/or the anonymous referee who read the preliminary version of the paper and pointed out many inaccuracies.

\bibliographystyle{alpha}

\begin{thebibliography}{BMRR10}

\bibitem[B{\'a}r82]{MR676720}
Imre B{\'a}r{\'a}ny.
\newblock A generalization of {C}arath{\'e}odory's theorem.
\newblock {\em Discrete Math.}, 40(2-3):141--152, 1982.

\bibitem[BF84]{MR771183}
E.~Boros and Z.~F{{\"u}}redi.
\newblock The number of triangles covering the center of an {$n$}-set.
\newblock {\em Geom. Dedicata}, 17(1):69--77, 1984.

\bibitem[BMN10]{MR2579699}
Boris Bukh, Ji{\v{r}}{\'{\i}} Matou{\v{s}}ek, and Gabriel Nivasch.
\newblock Stabbing simplices by points and flats.
\newblock {\em Discrete Comput. Geom.}, 43(2):321--338, 2010.

\bibitem[BMRR10]{MR2742969}
Abdul Basit, Nabil~H. Mustafa, Saurabh Ray, and Sarfraz Raza.
\newblock Improving the first selection lemma in {$\Bbb R^3$}.
\newblock In {\em Computational geometry ({SCG}'10)}, pages 354--357. ACM, New
  York, 2010.

\bibitem[Buk06]{MR2240753}
Boris Bukh.
\newblock A point in many triangles.
\newblock {\em Electron. J. Combin.}, 13(1):Note 10, 3 pp. (electronic), 2006.

\bibitem[Gro10]{MR2671284}
Mikhail Gromov.
\newblock Singularities, expanders and topology of maps. {P}art 2: {F}rom
  combinatorics to topology via algebraic isoperimetry.
\newblock {\em Geom. Funct. Anal.}, 20(2):416--526, 2010.

\bibitem[Kar12]{MR2891243}
Roman Karasev.
\newblock A simpler proof of the
  {B}oros--{F}{\"u}redi--{B}{\'a}r{\'a}ny--{P}ach--{G}romov theorem.
\newblock {\em Discrete Comput. Geom.}, 47(3):492--495, 2012.

\bibitem[KMS12]{MR2946458}
Daniel Kr{\'a}l', Luk{\'a}{\v{s}} Mach, and Jean-S{{\'e}}bastien
  Sereni.
\newblock A new lower bound based on {G}romov's method of selecting heavily
  covered points.
\newblock {\em Discrete Comput. Geom.}, 48(2):487--498, 2012.

\bibitem[Mat03]{matousek2003using}
Ji{\v{r}}{\'\i} Matou{\v{s}}ek.
\newblock {\em Using the {B}orsuk--{U}lam theorem: lectures on topological
  methods in combinatorics and geometry}.
\newblock Springer, 2003.

\bibitem[MW11]{matousek2011gromov}
Ji{\v{r}}{\'\i} Matou{\v{s}}ek and Uli Wagner.
\newblock On {G}romov's method of selecting heavily covered points.
\newblock {\em arXiv preprint arXiv:1102.3515}, 2011.

\bibitem[Wag03]{wagner2003k}
Ulrich Wagner.
\newblock {\em On k-sets and applications}.
\newblock PhD thesis, Swiss Federal Institute of Technology, ETH Z\"{u}rich,
  2003.

\end{thebibliography}

\end{document}